\documentclass[10pt,twoside]{article}
\usepackage{mathrsfs}
\usepackage{amssymb}
\usepackage{amsmath}
\usepackage[all]{xy}
\usepackage{amsthm}
\numberwithin{equation}{section}

\newtheorem{theorem}{Theorem}[section]
\newtheorem{proposition}[theorem]{Proposition}
\newtheorem{definition}[theorem]{Definition}

\newtheorem{lemma}[theorem]{Lemma}
\newtheorem{example}[theorem]{Example}

\newtheorem{corollary}[theorem]{Corollary}

\newcommand{\edge}{\ar@{-}}

\newcommand{\pf}{\noindent\begin {proof}}
\newcommand{\epf}{\end{proof}}

\newcommand{\Hom}{\mbox{\rm Hom}}

\newcommand{\Ima}{\mbox{\rm Im}}

\def\Im{\mathop{\rm Im}\nolimits}
\def\Ker{\mathop{\rm Ker}\nolimits}
\def\Coker{\mathop{\rm Coker}\nolimits}

\def\La{\mathop{\rm \Lambda}\nolimits}

\def\findim{\mathop{\rm fin.dim}\nolimits}

\def\mod{\mathop{\rm mod}\nolimits}

\def\pd{\mathop{\rm pd}\nolimits}
\def\max{\mathop{\rm max}\nolimits}
\def\min{\mathop{\rm min}\nolimits}
\def\sup{\mathop{\rm sup}\nolimits}
\def\inf{\mathop{\rm inf}\nolimits}
\def\add{\mathop{\rm add}\nolimits}

\def\gldim{\mathop{\rm gl.dim}\nolimits}
\def\rgldim{\mathop{\rm r.gl.dim}\nolimits}
\def\findim{\mathop{\rm fin.dim}\nolimits}

\def\rad{\mathop{{\rm rad}}\nolimits}
\def\top{\mathop{{\rm top}}\nolimits}

\def\dim{\mathop{\rm dim}\nolimits}

\def\Hom{\mathop{\rm Hom}\nolimits}

\def\sup{\mathop{\rm sup}\nolimits}
\def\lim{\mathop{\underrightarrow{\rm lim}}\nolimits}

\def\gen{\mathop{\rm gen}\nolimits}

\def\rank{\mathop{\rm rank}\nolimits}

\def\End{\mathop{\rm End}\nolimits}

\def\findim{\mathop{\rm fin.dim}\nolimits}

\def\mod{\mathop{\rm mod}\nolimits}

\def\pd{\mathop{\rm pd}\nolimits}
\def\max{\mathop{\rm max}\nolimits}
\def\min{\mathop{\rm min}\nolimits}
\def\sup{\mathop{\rm sup}\nolimits}
\def\inf{\mathop{\rm inf}\nolimits}
\def\add{\mathop{\rm add}\nolimits}

\def\gldim{\mathop{\rm gl.dim}\nolimits}
\def\findim{\mathop{\rm fin.dim}\nolimits}

\def\rad{\mathop{{\rm rad}}\nolimits}

\def\top{\mathop{{\rm top}}\nolimits}

\def\dim{\mathop{\rm dim}\nolimits}

\def\Hom{\mathop{\rm Hom}\nolimits}

\def\sup{\mathop{\rm sup}\nolimits}
\def\lim{\mathop{\underrightarrow{\rm lim}}\nolimits}

\def\gen{\mathop{\rm gen}\nolimits}

\def\rank{\mathop{\rm rank}\nolimits}

\def\End{\mathop{\rm End}\nolimits}
\def\repdim{\mathop{\rm rep.dim}\nolimits}
\def\wresoldim{\mathop{\rm w.resol.dim}\nolimits}

\def\A{\mathop{\rm \mathcal{A}}\nolimits}
\def\B{\mathop{\rm \mathcal{B}}\nolimits}
\def\C{\mathop{\rm \mathcal{C}}\nolimits}

\def\F{\mathop{\rm \mathcal{F}}\nolimits}
\def\S{\mathop{\rm \mathcal{S}}\nolimits}
\def\T{\mathop{\rm \mathcal{T}}\nolimits}

\def\P{\mathop{\rm \mathcal{P}}\nolimits}

\def\LL{\mathop{\rm LL}\nolimits}
\def\size{\mathop{\rm {\textbf{size}}}\nolimits}
\def\rank{\mathop{\rm {\textbf{rank}}}\nolimits}

\def\A{\mathop{\rm \mathcal{A}}\nolimits}
\def\B{\mathop{\rm \mathcal{B}}\nolimits}
\def\C{\mathop{\rm \mathcal{C}}\nolimits}
\def\T{\mathop{\rm \mathcal{T}}\nolimits}

\def\P{\mathop{\rm \mathcal{P}}\nolimits}

\def\gen.dim{\mathop{\rm gen.dim}\nolimits}
\def\size{\mathop{\rm \textbf{size}}\nolimits}
\def\rank{\mathop{\rm {\textbf{rank}}}\nolimits}

\title{ \bf The Extension Dimension of Abelian Categories 
}
\vspace{0.2cm}

\author{Junling Zheng$^{a}$, Xin Ma$^{b}$, Zhaoyong Huang$^{a,\,}$\thanks{E-mail address: zjlshuxue@163.com,\, maxin0719@126.com,\, huangzy@nju.edu.cn} \\
{\it \footnotesize $^{a}$Department of Mathematics, Nanjing University, Nanjing 210093, Jiangsu Province, P.R. China;}\\
{\it \footnotesize $^{b}$College of Science, Henan University of Engineering, Zhengzhou 451191, Henan Province, P.R. China}}

\date{ }

\begin{document}

\baselineskip=16pt


\maketitle

\begin{abstract}
Let $\A$ be an abelian category having enough projective objects and enough injective objects.
We prove that if $\A$ admits an additive generating object, then the extension dimension and the weak resolution dimension
of $\A$ are identical, and they are at most the representation dimension of $\A$ minus two. By using it,
for a right Morita ring $\La$, we establish the relation between the extension dimension
of the category $\mod \La$ of finitely generated right $\Lambda$-modules and the representation dimension as well as the
right global dimension of $\Lambda$. In particular, we give an upper bound for the extension dimension of $\mod \Lambda$
in terms of the projective dimension of certain class of simple right $\Lambda$-modules and the radical layer length of $\Lambda$.
In addition, we investigate the behavior of the extension dimension under some ring extensions and recollements.
\end{abstract}

\pagestyle{myheadings}
\markboth{\rightline {\scriptsize  J. L. Zheng, X. Ma, Z. Y. Huang\emph{}}}
         {\leftline{\scriptsize  The Extension Dimension of Abelian Categories }}


\section{Introduction} 

Following the work of Bondal and Van den Bergh \cite{BB03},
Rouquier introduced in \cite{RR08D} the dimension of triangulated categories, which is an invariant that measures
how quickly the category can be built from one object. This dimension plays
an important role in representation theory. For example, it can be used to compute the representation dimension of
artin algebras (\cite{RR06R, OS09L}). Let $\Lambda$ be an artin algebra and $\mod \Lambda$ the category of finitely
generated right $\Lambda$-modules.
Rouquier proved that the dimension of the bounded derived category
of $\mod \Lambda$ is at most $ \LL(\Lambda)-1$, where
$\LL(\Lambda)$ is the Loewy length of $\Lambda$, and this dimension
is at most the global dimension $\gldim\Lambda$ of $\Lambda$ if $\Lambda$ is a finite dimensional algebra over a perfect field
(\cite[Proposition 7.37 and Remark 7.26]{RR08D}).

As an analogue of the dimension of triangulated categories, the (extension) dimension $\dim\A$ of an abelian category $\A$
was introduced by Beligiannis in \cite{BA08S}, also see \cite{DHLTR14T}. Let $\Lambda$ be an artin algebra. Note that
the representation dimension of $\Lambda$ is at most two (that is, $\Lambda$ is of finite representation type) if and only if
$\dim\mod \Lambda=0$ (\cite{BA08S}). So, like the representation dimension of $\Lambda$, the extension dimension
$\dim\mod \Lambda$ is also an invariant that measures how far $\Lambda$ is from having finite representation type.
It was proved in \cite{BA08S} that $\dim\mod \Lambda\leq\LL(\Lambda)-1$, which is a semi-counterpart of the above result of Rouquier.
On the other hand, Iyama introduced in \cite{Iya03R} the weak resolution dimension of $\Lambda$ (see also \cite{OS09L}).
It is easy to see that the weak resolution dimension of $\Lambda$ is at most the representation dimension of $\Lambda$
minus two. Based on these works, in this paper we will study further properties of the extension dimension of abelian
categories, especially module categories. The paper is organized as follows.

In Section 2, we give some terminology and some preliminary results.

In Section 3, we investigate the relationship between the extension dimension and some other homological invariants.
Let $\A$ be an abelian category having enough projective objects and enough injective objects. We prove that
if $\A$ admits an additive generating object, then $\dim\A$ and the weak resolution dimension of $\A$ are identical,
and they are at most the representation dimension of $\A$ minus two.
For a ring $\Lambda$, we use $\rgldim\Lambda$ to denote the right global dimension of $\Lambda$.
As applications, we get that for a right Morita ring $\Lambda$, $\dim\mod\Lambda\leq\rgldim\Lambda$
(which is the other semi-counterpart of the result of Rouquier) and $\dim\mod\Lambda$ is at most the representation dimension of $\La$
minus two; and we also get that $\dim\mod \Lambda=n-1$ for the exterior algebra
$\Lambda$ of $k^n$, where $k$ is a field. In addition, we establish the relation between $\dim\mod \Lambda$ and
the finitistic dimension of $\Lambda$. Finally, we give an upper bound for $\dim\mod \Lambda$
in terms of the projective dimension of certain class of simple right $\Lambda$-modules and
the radical layer length of $\Lambda$, such that both $\gldim\Lambda$ and $\LL(\Lambda)-1$ are properly special cases of this upper bound.

In Section 4, we study the behavior of the extension dimension under ring extensions. Let $\Gamma\supseteq \Lambda$ be artin algebras.
We prove that $\dim \mod \Lambda=\dim \mod \Gamma$ if $\Gamma \geq \Lambda$ is an excellent extension, and that $\dim \mod \Lambda\leq
\dim \mod \Gamma+2$ if $\Gamma \geq \Lambda$ is a left idealized extension. We also prove that if $\Lambda$ and $\Gamma$ are
separably equivalent artin algebras, then $\dim \mod \Lambda=\dim \mod \Gamma$.

Let $\mathcal{A},\mathcal{B},\mathcal{C}$ be abelian categories and
$$\xymatrix{\mathcal{A}\ar[rr]!R|{i_{*}}&&\ar@<-2ex>[ll]!R|{i^{*}}\ar@<2ex>[ll]!R|{i^{!}}\mathcal{B}
\ar[rr]!L|{j^{*}}&&\ar@<-2ex>[ll]!L|{j_{!}}\ar@<2ex>[ll]!L|{j_{*}}\mathcal{C}}$$
a recollement. In Section 5, we prove that if either $i^{!}$ or $i^{*}$ is exact, then
$\max\{\dim \mathcal{A},\dim \mathcal{C}\} \leq \dim \mathcal{B}\leq \dim \mathcal{A} +\dim \mathcal{C}+1$.

\section{Preliminaries}

Let $\mathcal{A}$ be an abelian category.
The designation \emph{subcategory} will be used for full and additive subcategories of
$\mathcal{A}$ which are closed under isomorphisms and the word \emph{functor}
will mean an additive functor between additive categories.
For a subclass $\mathcal{U}$ of $\mathcal{A}$, we use $\add \mathcal{U}$ to
denote the subcategory of $\mathcal{A}$ consisting of
direct summands of finite direct sums of objects in $\mathcal{U}$.

Let $\mathcal{U}_1,\mathcal{U}_2,\cdots,\mathcal{U}_n$ be subcategories of $\mathcal{A}$. Define
$$\mathcal{U}_1\diamond \mathcal{U}_2:={\add}\{A\in \mathcal{A}\mid {\rm there \;exists \;an\; sequence \;}
0\rightarrow U_1\rightarrow  A \rightarrow U_2\rightarrow 0\ {\rm in}\ \mathcal{A}\ {\rm with}\; U_1 \in \mathcal{U}_1 \;{\rm and}\;
U_2 \in \mathcal{U}_2\}.$$
By \cite[Proposition 2.2]{DHLTR14T}, the operator $\diamond$ is associative, that is,
$(\mathcal{U}_{1}\diamond\mathcal{U}_{2})\diamond\mathcal{U}_{3} =\mathcal{U}_{1}\diamond(\mathcal{U}_{2}\diamond\mathcal{U}_{3}).$
The category $\mathcal{U}_{1}\diamond  \mathcal{U}_{2}\diamond \dots \diamond\mathcal{U}_{n}$
can be inductively described as follows
\begin{align*}
\mathcal{U}_{1}\diamond  \mathcal{U}_{2}\diamond \dots \diamond\mathcal{U}_{n}:=
\add \{A\in \mathcal{A}\mid {\rm there \;exists \;an\; sequence}\
0\rightarrow U\rightarrow  A \rightarrow V\rightarrow 0  \\{\rm in}\ \mathcal{A}\ {\rm with}\; U \in \mathcal{U}_{1} \;{\rm and}\;
V \in  \mathcal{U}_{2}\diamond \dots \diamond\mathcal{U}_{n}\}.
\end{align*}
For a subclass $\mathcal{U}$ of $\mathcal{A}$, set
$\langle\mathcal{U}\rangle_{0}=0$, $\langle\mathcal{U}\rangle_{1}=\add\mathcal{U}$,
$\langle\mathcal{U}\rangle_{n}=\langle\mathcal{U}\rangle_1\diamond \langle\mathcal{U}\rangle_{n-1}$ for any $n\geq 2$,
and $\langle\mathcal{U}\rangle_{\infty}=\mathop{\bigcup}_{n\geq 0}\langle\mathcal{U}\rangle_{n}$ (\cite{BA08S}).
Note that $\langle\mathcal{U}\rangle_{n}=\langle\langle\mathcal{U}\rangle_{1}\rangle_{n}$.
If $T$ is an object in $\mathcal{A}$ we write $\langle T\rangle_{n}$ instead of $\langle \{T \}\rangle_{n}$.

Throughout this paper, by convention, it is assumed that $\inf \emptyset=+\infty$ and $\sup \emptyset =-\infty.$

\begin{definition}\label{def-2.1}
{\rm (\cite{DHLTR14T})
For any subcategory $\mathcal{X}$ of $\mathcal{A}$, define
$$\size_{\mathcal{A}}\mathcal{X}:=\inf\{n\geq 0\mid\mathcal{X}\subseteq\langle T\rangle_{n+1}\ {\rm with}\ T\in\mathcal{A}\},$$
$$\rank_{\mathcal{A}}\mathcal{X}:=\inf\{n\geq 0\mid\mathcal{X}=\langle T\rangle_{n+1}\ {\rm with}\ T\in\mathcal{A}\}.$$
The {\bf extension dimension} $\dim \mathcal{A}$ of $\mathcal{A}$ is defined to be $\dim \mathcal{A}:=\rank_{\mathcal{A}}\mathcal{A}$.}
\end{definition}

It is easy to see that $\dim \mathcal{A}=\rank_{\mathcal{A}}\mathcal{A}=\size_{\mathcal{A}}\mathcal{A}$.
We also have the following easy and useful observations.

\begin{proposition}\label{prop-2.2}
Let $\mathcal{U}_{1}$ and $\mathcal{U}_{2}$ be subcategories of $\mathcal{A}$ with $\mathcal{U}_{1}\subseteq\mathcal{U}_{2}$. Then we have
\begin{itemize}
\item[(1)] If $\mathcal{V}_{1}$ and $\mathcal{V}_{2}$ are subcategories of $\mathcal{A}$ with $\mathcal{V}_{1}\subseteq\mathcal{V}_{2}$, then
$\mathcal{U}_{1}\diamond\mathcal{V}_{1}\subseteq\mathcal{U}_{2}\diamond\mathcal{V}_{2}$;
\item[(2)] $\langle\mathcal{U}_{1}\rangle_{n}\subseteq\langle\mathcal{U}_{2}\rangle_{n}$ for any $n\geq 1$;
\item[(3)] $\langle\mathcal{U}_{1}\rangle_{n}\subseteq\langle\mathcal{U}_{1}\rangle_{n+1}$ for any $n\geq 1$;
\item[(4)] ${\bf size}_{\mathcal{A}}\mathcal{U}_{1} \leq {\bf size}_{\mathcal{A}}\mathcal{U}_{2}$.
\end{itemize}
\end{proposition}

For two subcategories $\mathcal{U},\mathcal{V}$ of $\mathcal{A}$,
we set $\mathcal{U}\oplus \mathcal{V}:=\{U \oplus V\mid U\in\mathcal{U}\ {\rm and}\ V\in\mathcal{V} \}$.
Note that if $\mathcal{U}$ is closed under finite direct sums, then $\mathcal{U}\oplus \mathcal{U}= \mathcal{U}$.

\begin{corollary}\label{cor-2.3}
For any $T_{1},T_{2}\in \mathcal{A}$ and $m,n\geq 1$, we have
\begin{itemize}
\item[(1)] $\langle T_{1}\rangle_{m}\diamond \langle T_{2}\rangle_{n}\subseteq \langle T_{1}\oplus T_{2}\rangle_{m+n}$;
\item[(2)] $\langle T_{1}\rangle_{m}\oplus \langle T_{2}\rangle_{n}\subseteq \langle T_{1}\oplus T_{2}\rangle_{\max\{m,n\}}$.
\end{itemize}
\end{corollary}

\begin{proof}
Since $\langle T_{1}\rangle_{1} \subseteq \langle T_{1}\oplus T_{2}\rangle_{1}$, we have
$\langle T_{1}\rangle_{m}\subseteq \langle T_{1}\oplus T_{2}\rangle_{m}$ by Proposition \ref{prop-2.2}(2).
Similarly, $\langle T_{2}\rangle_{n}\subseteq \langle T_{1}\oplus T_{2}\rangle_{n}$.
Thus we have

(1) $\langle T_{1}\rangle_{m}\diamond \langle T_{2}\rangle_{n}\subseteq
\langle T_{1}\oplus T_{2}\rangle_{m}\diamond \langle T_{1}\oplus T_{2}\rangle_{n}
=\langle T_{1}\oplus T_{2}\rangle_{m+n}$.

(2) $\langle T_{1}\rangle_{m}\oplus \langle T_{2}\rangle_{n}\subseteq
\langle T_{1}\oplus T_{2}\rangle_{m}\oplus \langle T_{1}\oplus T_{2}\rangle_{n}
=\langle T_{1}\oplus T_{2}\rangle_{\max\{m,n\}}$ by Proposition \ref{prop-2.2}(3).
\end{proof}


We need the following fact.

\begin{lemma}\label{lem-2.4}
Let $F:\mathcal{A}\rightarrow \mathcal{B}$ be an exact functor of abelian categories.
Then $F(\langle T\rangle_{n})\subseteq\langle F(T)\rangle_{n}$ for any $T\in \mathcal{A}$ and $n\geq 1$.
\end{lemma}

\begin{proof}
We proceed by induction on $n$.
Let $X\in F(\langle T\rangle_{1})$. Then $X=F(Y)$ for some $Y\in\langle T\rangle_{1}(=\add T)$. Since
$Y\oplus Z\cong T^{l}$ for some $Z\in \mathcal{A}$ and $l\geq 1$, we have
$$X\oplus F(Z)=F(Y)\oplus F(Z)\cong F(Y\oplus Z)\cong F(T^{l})\cong F(T)^{l}.$$
So $X\in\langle F(T)\rangle_{1}$ and $F(\langle T\rangle_{1})\subseteq \langle F(T)\rangle_{1}$.
The case for $n=1$ is proved.

Now let $X\in F(\langle T\rangle_{n})$ with $n\geq 2$. Then $X=F(Y)$ for some $Y\in \langle T\rangle_{n}$ and
there exists an exact sequence
$$0\longrightarrow Y_{1} \longrightarrow Y\oplus Y'\longrightarrow Y_{2}\longrightarrow 0$$
in $\mathcal{A}$ with $Y_{1}\in \langle T\rangle_{1}$, $Y_{2}\in \langle T\rangle_{n-1}$ and $Y'\in\langle T\rangle_{n}$.
Since $F$ is exact, we get the following exact sequence
$$0\longrightarrow F(Y_{1}) \longrightarrow F(Y)\oplus F(Y')\longrightarrow F(Y_{2})\longrightarrow 0.$$
By the induction hypothesis, $F(Y_{1})\in F(\langle T\rangle_{1})\subseteq\langle F(T)\rangle_{1}$ and
$F(Y_{2})\in F(\langle T\rangle_{n-1})\subseteq\langle F(T)\rangle_{n-1}$. It follows that
$$X=F(Y)\in\langle F(Y_{1})\rangle_{1}\diamond\langle F(Y_{2})\rangle_{1}\subseteq
\langle F(T)\rangle_{1}\diamond\langle F(T)\rangle_{n-1}=\langle F(T)\rangle_{n}$$
and  $F(\langle T\rangle_{n})\subseteq\langle F(T)\rangle_{n}$.
\end{proof}

\section{Relations with some homological invariants}

In this section, $\mathcal{A}$ is an abelian category.

%
%
%

\begin{definition}{\rm (cf. \cite{Iya03R, OS09L})}\label{def-3.1}
{\rm Let $M\in\mathcal{A}$.
The {\bf weak $M$-resolution dimension} of an object $X$ in $\mathcal{A}$, denoted by $M$-$\wresoldim X$,
is defined as $\inf\{i\geq 0\mid$ there exists an exact sequence
$$0 \longrightarrow M_{i} \longrightarrow M_{i-1}\longrightarrow\cdots \longrightarrow M_{0}\longrightarrow X\longrightarrow 0$$
in $\mathcal{A}$ with all $M_{j}$ in $\add M\}$.
The {\bf weak $M$-resolution dimension} of $\mathcal{A}$, $M$-$\wresoldim \mathcal{A}$, is defined as
$\sup\{M$-$\wresoldim X \mid X\in \mathcal{A}\}$.
The {\bf weak resolution dimension} of $\mathcal{A}$ is denoted by  $\wresoldim \mathcal{A}$
and defined as $\inf\{M$-$\wresoldim \mathcal{A}\mid M\in \mathcal{A}\}$.}
\end{definition}


Let $X\in\A$. Suppose there exists a monomorphism $f: X\longrightarrow E$ in $\A$ such that $E$
is an injective object in $\A$. Then we write $\Omega^{-1}(X)=:\Coker f$ if $f$ is right minimal, i.e. if
$f$ is an injective envelope of $X$. Dually,
if $g: P\longrightarrow X$ is a right minimal epimorphism in $\A$ such that $P$
is a projective object in $\A$, then we write $\Omega^{1}(X)=:\Ker f$.
Additionally, define $\Omega^{0}$ as the identity functor in $\A$.
Inductively, for any $n\geq 2$,
we write $\Omega^{n}(X):=\Omega^{1}(\Omega^{n-1}(X))$ and $\Omega^{-n}(X):=\Omega^{-1}(\Omega^{-(n-1)}(X))$.

\begin{lemma} {\rm (\cite[Lemma 3.3]{XuD16I})}\label{lem-3.2}
If $\A$ has enough projective objects and enough injective objects, then
for any exact sequence
$$0\longrightarrow X_{1} \longrightarrow  X_{2} \longrightarrow  X_{3} \longrightarrow 0$$
in $\A$, we have the following exact sequences
$$0\longrightarrow \Omega^{1}(X_{3})\longrightarrow X_{1}\oplus P \longrightarrow X_{2}\longrightarrow 0,$$
$$0 \longrightarrow X_{2}\longrightarrow E\oplus X_{3}\longrightarrow \Omega^{-1}(X_{1})\longrightarrow 0,$$
where $P$ is projective and $E$ is injective in $\mathcal{A}$.
\end{lemma}
%
%

Using Lemma \ref{lem-3.2}, we get the following lemma, which is a dual of \cite[Lemma 5.8]{DHLTR14T}.

\begin{lemma}\label{lem-3.3}
If $\A$ has enough injective objects and
$$0\longrightarrow M_{n}\longrightarrow
\cdots \longrightarrow  M_{1} \longrightarrow  M_{0} \longrightarrow X \longrightarrow0,$$
is an exact sequence in $\A$ with $n\geq 0$, then
$$X\in  \langle M_{0} \rangle_{1}\diamond\langle \Omega^{-1}(M_{1})
 \rangle_{1}\diamond\cdots\diamond\langle \Omega^{-n}(M_{n}) \rangle_{1}
 \subseteq\langle \oplus_{i=0}^{n}\Omega^{-i}(M_{i}) \rangle_{n+1}.$$
\end{lemma}

\emph{Remark.} Note that if $\A$ has enough injectives and
 $X\in \langle Y_{1} \rangle_{1} \diamond  \langle Y_{2} \rangle_{1}$,
 then $\Omega^{-1}(X)\in \langle \Omega^{-1}(Y_{1}) \rangle_{1} \diamond\langle \Omega^{-1}(Y_{2}) \rangle_{1}$.
 This fact is a sequence of the Horseshoe Lemma and is used to prove Lemma \ref{lem-3.3}. This statement and its corresponding dual
 version will be throughout this paper.
\subsection{Representation and global dimensions}

For a subclass $\mathcal{X}$ of $\mathcal{A}$,
recall that a sequence $\mathbb{S}$ in $\mathcal{A}$ is called {\bf $\Hom_{\mathcal{A}}(\mathcal{X},-)$-exact}
(resp. {\bf $\Hom_{\mathcal{A}}(-,\mathcal{X})$-exact}) if $\Hom_{\mathcal{A}}(X,\mathbb{S})$
(resp. $\Hom_{\mathcal{A}}(\mathbb{S},X)$) is exact for any $X\in \mathcal{X}$.

\begin{definition}\label{def-3.4}
{\rm (\cite{AM71R, EKHT04R, RR06R})
The {\bf representation dimension} $\repdim \mathcal{A}$ of $\mathcal{A}$
is the smallest integer $i\geq 2$ such that there exists $M\in \mathcal{A}$ satisfying the property that for any
$X\in \mathcal{A}$,
\begin{enumerate}
\item[(1)] there exists a $\Hom_{\mathcal{A}}(\add M,-)$-exact exact sequence
$$0 \longrightarrow M_{i-2} \longrightarrow M_{i-3}\longrightarrow\cdots \longrightarrow M_{0}\longrightarrow X\longrightarrow 0$$
in $\mathcal{A}$ with all $M_{j}$ in $\add M$; and
\item[(2)] there exists a $\Hom_{\mathcal{A}}(-,\add M)$-exact exact sequence
$$0 \longrightarrow X\longrightarrow N_{0} \longrightarrow N_{1}\longrightarrow\cdots\longrightarrow N_{i-2}\longrightarrow 0$$
in $\mathcal{A}$ with all $N_{j}$ in $\add M$.
\end{enumerate}}
\end{definition}

We call $A\in\A$ an {\bf additive generating object} if $\add A$ is a generator for $\A$.
It is trivial that if $A\in\A$ is an additive generating object, then all projective objects in $\A$ are in $\add A$.

\begin{theorem}\label{thm-3.5}
Assume that $\A$ admits an additive generating object $A$.
If $\A$ has enough projective objects and enough injective objects, then
$$\wresoldim \A=\dim \A\leq \repdim \A-2.$$
\end{theorem}

\begin{proof}
It is trivial that $\wresoldim \A\leq \repdim \A-2$.

Assume that $\dim\A=n$ and $T\in\A$ such that $\A=\langle T \rangle_{n+1}$. Let $X\in\A$. Then we have an exact sequence
\begin{align*}
0\longrightarrow  X_{1}\longrightarrow X \longrightarrow X_{2}\longrightarrow 0
\end{align*}
in $\A$ with $ X_{1}\in\langle T \rangle_{1}$ and $X_{2}\in\langle T \rangle_{n}$. Set $M:=\oplus_{i=0}^{n}\Omega^{i}(T)\oplus A$.
We will prove $M$-$\wresoldim X\leq n$ by induction on $n$. The case for $n=0$ is trivial. If $n=1$,
then $T$-$\wresoldim X_{2}=0$ and $M$-$\wresoldim\Omega^{1}(X_{2})=0$.
By Lemma \ref{lem-3.2}, we have an exact sequence
$$0\longrightarrow \Omega^{1}(X_{2})\longrightarrow X_{1}\oplus P \longrightarrow X\longrightarrow 0$$
in $\A$ with $P$ projective.  So $M$-$\wresoldim X\leq 1$. Now suppose $n\geq 2$. By the induction hypothesis,
we have $(\oplus_{i=0}^{n-1}\Omega^{i}(T)\oplus A)$-$\wresoldim X_{2}\leq n-1$, hence $M$-$\wresoldim\Omega^{1}(X_{2})\leq n-1$.
It follows that $M$-$\wresoldim X\leq n$. Thus we have $\wresoldim\A\leq n$.

Conversely, assume that $\wresoldim\A=n$ and $T\in\A$ such that for any $X\in\A$, there exists an exact sequence
$$0\longrightarrow M_{n}\longrightarrow
\cdots \longrightarrow  M_{1} \longrightarrow  M_{0} \longrightarrow X \longrightarrow 0$$
in $\A$ with all $M_{i}$ in $\add T$. By Lemma \ref{lem-3.3}, we have that
$X\in \langle \oplus_{i=0}^{n}\Omega^{-i}(M_{i}) \rangle_{n+1}\subseteq \langle \oplus_{i=0}^{n}\Omega^{-i}(T) \rangle_{n+1}$
and $\A \subseteq  \langle \oplus_{i=0}^{n}\Omega^{-i}(T) \rangle_{n+1}$.
It follows that $\A = \langle \oplus_{i=0}^{n}\Omega^{-i}(T) \rangle_{n+1}$. Thus we have $\dim \A\leq n$.
\end{proof}

For a ring $\Lambda$, we use $\mod \Lambda$ to denote the category of finitely generated right $\Lambda$-modules,
and we write $\repdim \La:=\repdim \mod \La$ if $\mod\La$ is an abelian category.
Recall from \cite{H94} that a ring $\Lambda$ is called {\bf right Morita} if there exist a ring $\Gamma$ and
a Morita duality from $\mod \Lambda$ to $\mod \Gamma^{op}$. It is known that a ring $\Lambda$ is right Morita
if and only if it is right artinian and there exists a finitely generated injective cogenerator for the category of right $\Lambda$-modules
(\cite[p.165]{H94}). The class of right Morita rings includes right pure-semisimple rings and artin algebras.
For any right noetherian ring $\Lambda$,
it is clear that $\wresoldim \mod \Lambda\leq \rgldim \Lambda$.
 So, as an immediate consequence of Theorem \ref{thm-3.5},
we have the following

\begin{corollary}\label{cor-3.6}
If $\Lambda$ is a right Morita ring, then
$$\wresoldim\mod \La=\dim \mod \Lambda\leq \min\{\rgldim \Lambda,\repdim \La-2\}.$$
\end{corollary}

Let $\Lambda$ be an artin algebra. Recall that $\Lambda$ is called {\bf $n$-Gorenstein} if its left and right self-injective dimensions
are at most $n$. Let $\P$ be the subcategory of $\mod \Lambda$ consisting of projective modules.
A module $G\in\mod \Lambda$ is called {\bf Gorenstein projective} if there exists a $\Hom_{\Lambda}(-,\P)$-exact exact sequence
$$\cdots \to P_1\to P_0 \to P^0\to P^1\to \cdots$$
in $\mod \Lambda$ with all $P_i,P^i$ in $\P$ such that $G\cong\Im(P_0 \to P^0)$. Recall from \cite{BA11} that $\Lambda$ is said to be of
{\bf finite Cohen-Macaulay type} ({\bf finite CM-type} for short) if there are only finitely many non-isomorphic indecomposable
Gorenstein projective modules in $\mod \Lambda$.

\begin{corollary}\label{cor-3.7}
If $\Lambda$ is an $n$-Gorenstein artin algebra of finite CM-type, then $\dim \mod \Lambda\leq n$.
\end{corollary}

\begin{proof}
Let $M\in\mod \Lambda$. Because $\Lambda$ is an $n$-Gorenstein artin algebra, we have an exact sequence
$$0\to H_n\to \cdots \to H_1\to H_0 \to M \to 0$$
in $\mod \Lambda$ with all $H_i$ Gorenstein projective by \cite[Theorem 1.4]{HH12}. Because $\Lambda$
is of finite CM-type, we may assume that $\{G_1,\cdots,G_n\}$ is the set of non-isomorphic indecomposable
Gorenstein projective modules in $\mod \Lambda$. Set $G:=\oplus_{i=0}^nG_i$. Then $G$-$\wresoldim M\leq n$
and $\wresoldim \mod\Lambda\leq n$. It follows from Theorem \ref{thm-3.5} that $\dim\mod\Lambda\leq n$.
\end{proof}

For small $\dim \mod \Lambda$, we have the following

\begin{corollary}\label{cor-3.8}
Let $\Lambda$ be an artin algebra. Then we have
\begin{itemize}
\item[(1)] {\rm (\cite[Example 1.6(i)]{BA08S})} $\repdim \Lambda\leq 2$ if and only if $\dim \mod \Lambda=0$;
\item[(2)] if $\repdim \Lambda=3$, then $\dim \mod \Lambda=1$.
\end{itemize}
\end{corollary}

\begin{proof}
(1) It is trivial by Corollary \ref{cor-3.6}.

(2) Let $\repdim \Lambda=3$. Then $\dim \mod \Lambda\geq 1$ by (1); and $\dim \mod \Lambda\leq \repdim \Lambda-2=1$ by Corollary \ref{cor-3.6}.
The assertion follows.
\end{proof}

For a field $k$ and $n\geq 1$, $\wedge(k^{n})$ is the exterior algebra of $k^n$.

\begin{corollary}\label{cor-3.9}
$\dim \mod \wedge(k^{n})=n-1$ for any $n\geq 1$.
\end{corollary}

\begin{proof}
By \cite[Thoerem 4.6]{Iya03R}, we have $\wresoldim \mod \wedge(k^{n})=n-1$.
It follows from Corollary \ref{cor-3.6} that $\dim \mod \wedge(k^{n})=n-1$.
\end{proof}




\subsection{Finitistic dimension}

{\bf From now on, $\Lambda$ is an artin algebra.} For a module $M$ in $\mod \Lambda$, $\pd M$ is the projective dimension of $M$.
Set $\mathcal{P}^{<\infty}:=\{M\in \mod \Lambda \mid \pd M<\infty\}$.
Recall that the {\bf finitistic dimension} $\findim \Lambda$ of $\Lambda$ is defined as
$\sup\{\pd M\mid M\in \mathcal{P}^{<\infty}\}$. It is an unsolved conjecture that
$\findim \Lambda<\infty$ for every artin algebra $\Lambda$. Igusa-Todorov introduced in \cite{IT05O} a powerful
function $\psi$ from $\mod \Lambda$ to non-negative integers to study the finiteness of $\findim \Lambda$.
The following lemma gives some useful properties of the Igusa-Todorov function $\psi$.

\begin{lemma}{\rm (\cite[Lemma 0.3 and Theorem 0.4]{IT05O})}\label{lem-3.10}
\begin{enumerate}
\item[(1)] For any $X,Y\in\mod\Lambda$, $\psi(X)\leq \psi(Y)$ if $\langle X \rangle_{1} \subseteq \langle Y \rangle_{1}$;
\item[(2)] if $ 0 \longrightarrow X_{1} \longrightarrow  X_{2}\longrightarrow X_{3}\longrightarrow 0$
is an exact sequence in $\mod \Lambda$ with $\pd X_{3}<\infty$, then
$\pd X_{3} \leq \psi(X_{1}\oplus X_{2})+1$.
\end{enumerate}
\end{lemma}

For any subcategory $\mathcal{X}$ of $\mod \Lambda$ and $n\geq 0$, set 
$\Omega^n(\mathcal{X}):=\{\Omega^n(M)\mid M\in \mathcal{X}\}$; in particular, $\Omega^0(\mathcal{X})=\mathcal{X}$.

\begin{proposition}\label{prop-3.11}
The following statements are equivalent.
\begin{enumerate}
\item[(1)] $\findim \Lambda <\infty$;
\item[(2)] there exists some $n\geq 0$ such that ${\bf size}_{\mod \Lambda}\Omega^{n}(\mathcal{P}^{<\infty})\leq 1$.
\end{enumerate}
\end{proposition}

\begin{proof}
$(1)\Rightarrow (2)$ If $\findim \Lambda=m <\infty$, then $\Omega^{m}(\mathcal{P}^{<\infty})\subseteq \langle \Lambda \rangle_{1}$
and $\size_{\mod \Lambda}\Omega^{m}(\mathcal{P}^{<\infty})=0$.

$(2)\Rightarrow (1)$ Let $\size_{\mod \Lambda}\Omega^{n}(\mathcal{P}^{<\infty})\leq 1$ with $n\geq 0$. Then
$\Omega^{n}(\mathcal{P}^{<\infty})\subseteq \langle T \rangle_{2}$ for some $T\in\mod \Lambda$.
Let $X\in \mathcal{P}^{<\infty} $. Then there exists an exact sequence
$$0 \longrightarrow T_{1} \longrightarrow  \Omega^{n}(X)\longrightarrow T_{2}\longrightarrow 0$$
in $\mod \Lambda$ with $T_{1},T_{2}\in \langle T \rangle_{1}$.
By Lemma \ref{lem-3.2}, we obtain the following exact sequence
$$ 0\longrightarrow \Omega^{1}(T_{2}) \longrightarrow T_{1}\oplus P \longrightarrow \Omega^{n}(X)\longrightarrow 0$$
with $P\in \langle \Lambda \rangle_{1}$. Then we have
\begin{align*}
\pd X&\leq\pd \Omega^{n}(X)+n\\
&\leq\psi( \Omega^{1}(T_{2}) \oplus T_{1}\oplus P )+1+n&({\rm by\ Lemma}\ \ref{lem-3.10}(2))\\
&\leq\psi( \Omega^{1}(T) \oplus T\oplus \Lambda)+1+n,&({\rm by\ Lemma}\ \ref{lem-3.10}(1))
\end{align*}
which implies $\findim \Lambda\leq \psi( \Omega^1(T) \oplus T\oplus \Lambda)+1+n$.
\end{proof}

By Proposition \ref{prop-3.11}, we have the following

\begin{corollary}\label{cor-3.12}
If $\dim \mod \Lambda\leq 1$, then $\findim \Lambda <\infty$.
\end{corollary}

\subsection{Igusa-Todorov algebras}

\begin{definition}\label{def-3.13}
{\rm (\cite{WJQ09F} and \cite[Lemma 3.6]{HZYSJX13E})
For an integer $n\geq 0$, $\Lambda$
is called {\bf ($n$-)Igusa-Todorov} if there exists $V\in\mod \Lambda$
such that for any $M\in\mod \Lambda$, there exists an exact sequence
$$0 \longrightarrow V_{1}\longrightarrow V_{0}\longrightarrow \Omega^{n}(M)\oplus P\longrightarrow 0$$
in $\mod \Lambda$ with $V_{1}$, $V_{0}\in\add V$ and $P$  projective; equivalently, there exists a module $V\in\mod \Lambda$
such that for any $M\in\mod \Lambda$, there exists an exact sequence
$$0 \longrightarrow V_{1}\longrightarrow V_{0}\longrightarrow \Omega^{n}(M)\longrightarrow 0$$
in $\mod \Lambda$ with $V_{1}$, $V_{0}\in\add V$.}
\end{definition}

The class of Igusa-Todorov algebras includes algebras with representation dimension
at most 3, algebras with radical cube zero, monomial algebras, left serial algebras and syzygy finite algebras (\cite{WJQ09F}).

\begin{theorem}\label{thm-3.14}
For any $n\geq 0$, the following statements are equivalent.
\begin{enumerate}
\item[(1)] $\Lambda$ is $n$-Igusa-Todorov;
\item[(2)] ${\bf size}_{\mod \Lambda}\Omega^{n}(\mod \Lambda)\leq 1$.
\end{enumerate}
\end{theorem}

\begin{proof}
$(1)\Rightarrow (2)$
Let $\Lambda$ be $n$-Igusa-Todorov and $X\in \Omega^{n}(\mod \Lambda)$. Then there exists $V\in\mod \Lambda$ such that the following sequence
$$0 \longrightarrow V_{1} \longrightarrow V_{0}  \longrightarrow  X\longrightarrow 0,$$
in $\mod \Lambda$ with $V_{1},V_{0}\in \add V$ is exact.
By Lemma \ref{lem-3.3}, Proposition \ref{prop-2.2}(1) and Corollary \ref{cor-2.3}(1), we have
$$X\in \langle  V_{0}\rangle_{1}\diamond \langle  \Omega^{-1}(V_{1})\rangle_{1}
\subseteq\langle  V\rangle_{1}\diamond \langle  \Omega^{-1}(V)\rangle_{1}
\subseteq \langle  V\oplus\Omega^{-1}(V)\rangle_{2}.$$
And then $\size_{\mod \Lambda}\Omega^{n}(\mod \Lambda)\leq 1$ by Definition \ref{def-2.1}.

$(2)\Rightarrow (1)$
Let $\size_{\mod \Lambda}\Omega^{n}(\mod \Lambda)\leq 1$ and $X\in\mod \Lambda$.
Then there exists $T\in \mod \Lambda$ such that the following sequence
$$0  \longrightarrow T_{1}\longrightarrow  \Omega^{n}(X)  \longrightarrow T_{2}\longrightarrow 0,$$
in $\mod \Lambda$ with $T_{1},T_{2}\in \langle T \rangle_{1}$ is exact.
By Lemma \ref{lem-3.2}, we obtain the following exact sequence
$$0 \longrightarrow \Omega^{1}(T_{2}) \longrightarrow T_{1}\oplus P\longrightarrow  \Omega^{n}(X)\longrightarrow 0$$
in $\mod \Lambda$ with $P$ projective. Since both $\Omega^{1}(T_{2})$ and $T_{1}\oplus P$ are in $\add (\Omega^{1}(T)\oplus T \oplus \Lambda)$,
we have that $\Lambda$ is $n$-Igusa-Todorov.
\end{proof}

The first assertion in the following proposition means that $\dim \mod \Lambda$ is an invariant for measuring how far $\Lambda$
is from being 0-Igusa-Todorov.

\begin{proposition}\label{prop-3.15}
\begin{enumerate}
\item[]
\item[(1)] $\Lambda$ is 0-Igusa-Todorov if and only if $\dim \mod \Lambda \leq 1$;
\item[(2)] if $\Lambda$ is $n$-Igusa-Todorov, then $\dim \mod \Lambda\leq n+1$.
\end{enumerate}
\end{proposition}

\begin{proof}
(1) It is trivial by Theorem \ref{thm-3.14}.

(2) Let $\Lambda$ be $n$-Igusa-Todorov and $X\in \Omega^{n}(\mod \Lambda)$. Then there exists $V\in\mod \Lambda$ such that the following sequence
$$0 \longrightarrow V_{2} \longrightarrow V_{1} \longrightarrow P_{n-1}\longrightarrow\cdots
\longrightarrow P_{1}\longrightarrow P_{ 0} \longrightarrow X\longrightarrow 0$$
in $\mod \Lambda$ with $V_{2},V_{1}\in \add V$ and all $P_{i}$ projective.
Thus $\wresoldim \mod \Lambda \leq n+1$, and therefore $\dim \mod \Lambda \leq n+1$ by Theorem \ref{thm-3.5}.
\end{proof}

Moreover, we have the following

\begin{corollary}\label{cor-3.16}
$\dim \mod \Lambda \leq 2$ if $\Lambda$ is in one class of the following algebras.
\begin{itemize}
\item[(1)] monomial algebras;
\item[(2)] left serial algebras;
\item[(3)] $\rad^{2n+1}\Lambda=0$ and $\Lambda/\rad^{n}\Lambda$ is representation finite;
\item[(4)] 2-syzygy finite algebras.
\end{itemize}
\end{corollary}

\begin{proof}
By \cite[Corollaries 2.6, 3.5 and Proposition 2.5]{WJQ09F}, these four classes of algebras are 1-Igusa-Todorov.
So the assertions follow from Proposition \ref{prop-3.15}.
\end{proof}

\subsection{$t_{\S}$-radical layer length}


We recall some notions from \cite{HLM2}.
Let $\C$ be a {\bf length-category}, that is, $\C$
is an abelian, skeletally small category and every object of $\C$ has a finite composition series.
We denote by $\End_{\mathbb{Z}}(\C)$ the category of all additive functors from
$\C$ to $\C$, and denote by $\rad$ the Jacobson radical lying in
$\End_{\mathbb{Z}}(\C)$.
Let $\alpha,\beta\in\End_{\mathbb{Z}}(\C) $ and $\alpha$ be a subfunctor
of $\beta$, we have the quotient functor $\beta/\alpha\in \End_{\mathbb{Z}}(\C)$
which is defined as follows.
\begin{enumerate}
\item[(1)] $(\beta/\alpha)(M):=\beta(M)/\alpha(M)$ for any $M\in \C$; and
\item[(2)] $(\beta/\alpha)(f)$ is the induced quotient morphism: for any $f\in\Hom_{\C}(M,N)$,
\[\xymatrix{
0 \ar[r]& \alpha(M)\ar[r]\ar[d]^{\alpha(f)}& \beta(M)\ar[r]\ar[d]^{\beta(f)}
& \beta (M)/ \alpha (M)\ar[r]\ar@{-->}[d]^{(\beta/ \alpha)(f)}&0\\
0 \ar[r]& \alpha(N)\ar[r]& \beta(N)\ar[r]& \beta (N)/ \alpha (N)\ar[r]&0.}\]
\end{enumerate}

For any $\alpha\in\End_{\mathbb{Z}}(\C)$, set the {\bf $\alpha$-radical functor} $F_{\alpha}:=\rad\circ \alpha$.
We define the following two classes
$$\F_{\alpha}:=\{ M\in \C\mid\alpha(M)=0\},\;\;\;\T_{\alpha}=\{ M\in \C\mid \alpha(M)\cong M\}.$$

\begin{definition}{\rm (\cite[Definition 3.1]{HLM2})\label{def-3.17}
For any $\alpha, \beta \in \End_{\mathbb{Z}}(\C)$,
the {\bf $(\alpha,\beta)$-layer length} of $M \in \C$, denoted by $\ell\ell_{\alpha}^{\beta}(M)$, is defined as
$\ell\ell_{\alpha}^{\beta}(M)=\inf\{ i \geq 0\mid \alpha \circ \beta^{i}(M)=0 \}$. Moreover,
$\ell\ell_{\alpha}^{\beta}$ goes from $\C$ to $\mathbb{N}\cup \{+\infty\}$.
And
the {\bf $\alpha$-radical layer length} $\ell\ell^{\alpha}:=\ell\ell_{\alpha}^{F_{\alpha}}$.}

\end{definition}


\begin{lemma}{\rm (\cite[Lemma 2.6]{ZH})}\label{lem-3.18}
Let $\alpha,\beta \in\End_{\mathbb{Z}}(\C) $.
For any $M\in \C$, if $\ell\ell_{\alpha}^{\beta}(M)=n$, then $\ell\ell_{\alpha}^{\beta}(M)=\ell\ell_{\alpha}^{\beta}(\beta^{i}(M))+i$
for any $0 \leq i\leq n$; in particular, if $\ell\ell^{\alpha}(M)=n$, then $\ell\ell^{\alpha}(F_{\alpha}^{n}(M))=0$.
\end{lemma}

Recall that a {\bf torsion pair} (or {\bf torsion theory}) for $\C$
is a pair of classes $(\T,\F)$ of objects in $\C$ satisfying the following conditions.
\begin{enumerate}
\item[(1)] $\Hom_{\C}(M,N)=0$ for any $M\in\T$ and $N\in\F$;
\item[(2)] an object $X \in \C$ is in $\T$ if $\Hom_{\C}(X,-)|_{\F}=0$;
\item[(3)] an object $Y\in\C$ is in $\F$ if $\Hom_{\C}(-,Y)|_{\T}=0$.
\end{enumerate}

Let $(\T,\F)$ be a torsion pair for $\C$. Recall that $t:={\rm Trace}_{\T}$ is the so called {\bf torsion radical}
attached to $(\T, \F)$. Then $t(M):=\Sigma\{\Im f\mid f\in\Hom_{\C}(T,M)$ with $T\in\T\}$
is the largest subobject of $M$ lying in $\T$.

For a subfunctor $\alpha \in\End_{\mathbb{Z}}(\C)$ of the identity functor $1_{\C}$ of $\C$, we write
$q_{\alpha}:=1_{\C}/\alpha$. The functor $q_{\alpha}$ lies in $\End_{\mathbb{Z}}(\C)$.
In this section, $\La$ is an artin algebra. Then $\mod \La$ is a length-category.
We use $\rad \La$ to denote the Jacobson radical of $\La$.
For a module $M$ in $\mod\La$, we use $\top M$ to denote the top of $M$.
Set $\pd M=-1$ if $M=0$. For a subclass $\mathcal{B}$ of $\mod \La$, the {\bf projective dimension} $\pd\mathcal{B}$ of
$\mathcal{B}$ is defined as
\begin{equation*}
\pd \B=
\begin{cases}
\sup\{\pd M\mid M\in \B\}, & \text{if} \;\; \B \neq\emptyset;\\
-1,&\text{if} \;\; \B =\emptyset.
\end{cases}
\end{equation*}
We use $\S^{<\infty}$ to denote the set of the simple modules in $\mod \La$ with finite projective dimension.

From now on, assume that $\S$ is a subset of $\S^{<\infty}$ and $\S'$ is the set of all the
others simple modules in $\mod \Lambda$. We write $\mathfrak{F}\,(\mathcal{S}):=\{M\in\mod\Lambda\mid$ there exists a finite chain
$$0=M_0\subseteq M_1\subseteq \cdots\subseteq M_m=M$$ of submodules of $M$
such that each quotient $M_i / M_{i-1}$ is isomorphic to some module in $\mathcal{S}\}$.
By \cite[Lemma 5.7 and Proposition 5.9]{HLM2}, we have that
$(\T_{\S}, \mathfrak{F}(\S))$ is a torsion pair, where
$$\T_{\S}=\{M \in \mod \La\mid\top M\in \add \S'\}.$$
We denote the torsion radical $t_{\S}={\rm Trace}_{\T_{\S}}$.
Then $t_{\S}(M)\in \T_{\S}$ and $q_{_{t_{\S}}}(M)\in\mathfrak{F}(\S)$ for any
$M\in \mod \La$. By \cite[Proposition 5.3]{HLM2}, we have
$$\mathfrak{F}(\S)=\{ M\in \mod \La \mid t_{\S}(M)=0\},$$
$$\T_{\S}=\{ M\in \mod \La\mid t_{\S}(M)= M\}.$$





\begin{theorem}\label{thm-3.19}
Let $\S$ be a subset of the set $\S^{<\infty}$ of all pairwise non-isomorphism simple $\La$-modules
with finite projective dimension. Then
$\dim \mod \La \leq \pd \S+\ell\ell^{t_{\S}}(\La)$.
\end{theorem}

\begin{proof}
Let $\ell\ell^{t_{\S}}(\La)=n$ and $\pd \S=\alpha$.

If $n=0$, that is, $t_{\S}(\La)=0$, then $\La\in \mathfrak{F}(\S)$,
which implies that $\S$ is the set of all simple modules. Thus $\S=\S^{<\infty}$
and $\gldim \Lambda=\alpha$. So the assertion follows from Corollary \ref{cor-3.6}.

Now let $n\geq 1$ and $M\in \mod \La$. Consider the following exact sequence
$$0 \longrightarrow \Omega^{\alpha+2}(M) \longrightarrow L_{\alpha+1}\longrightarrow\cdots
\longrightarrow L_{1}\longrightarrow L_{ 0} \longrightarrow M\longrightarrow 0$$
in $\mod \La$ with all $L_{i} $ projective. By Lemma \ref{lem-3.3}, we have
\begin{align*}
M\in &\langle L_{0} \rangle_{1}\diamond\langle \Omega^{-1}(L_{1})
\rangle_{1}\diamond\cdots\diamond\langle \Omega^{-\alpha-1}(L_{\alpha+1}) \rangle_{1}
\diamond\langle \Omega^{-\alpha-2}(\Omega^{\alpha+2}(M))\rangle_{1}\\
\subseteq &\langle\oplus_{i=0}^{-\alpha-1} \Omega^{i}(\Lambda) \rangle_{\alpha+2}
\diamond\langle \Omega^{-\alpha-2}(\Omega^{\alpha+2}(M))\rangle_{1}.
\end{align*}
We have the following exact sequences
\begin{align*}
0 \rightarrow t_{\S}(M)\rightarrow &M\rightarrow q_{t_{\S}}(M) \rightarrow 0,\\
0 \rightarrow t_{\S}(\Omega^{1} (t_{\S}(M)))\rightarrow &\Omega^{1} (t_{\S}(M))\rightarrow q_{t_{\S}}(\Omega^{1} (t_{\S}(M))) \rightarrow 0,\\
0 \rightarrow F_{t_{\S}}(\Omega^{1} (t_{\S}(M)))\rightarrow &t_{\S}(\Omega^{1} (t_{\S}(M)))\rightarrow \top t_{\S}(\Omega^{1} (t_{\S}(M))) \rightarrow 0,\\
0 \rightarrow t_{\S}F_{t_{\S}}(\Omega^{1} (t_{\S}(M)))\rightarrow &F_{t_{\S}}(\Omega^{1} (t_{\S}(M)))\rightarrow q_{t_{\S}}F_{t_{\S}}(\Omega^{1} (t_{\S}(M))) \rightarrow 0,\\
0 \rightarrow F^{2}_{t_{\S}}(\Omega^{1} (t_{\S}(M)))\rightarrow &t_{\S}F_{t_{\S}}(\Omega^{1} (t_{\S}(M)))\rightarrow \top t_{\S}F_{t_{\S}}(\Omega^{1} (t_{\S}(M))) \rightarrow 0,\\
&\cdots \cdots \cdots \cdots\\
0 \rightarrow t_{\S}F^{n-2}_{t_{\S}}(\Omega^{1} (t_{\S}(M)))\rightarrow &F^{n-2}_{t_{\S}}(\Omega^{1} (t_{\S}(M)))\rightarrow q_{t_{\S}}F^{n-2}_{t_{\S}}(\Omega^{1} (t_{\S}(M))) \rightarrow 0,\\
0 \rightarrow F^{n-1}_{t_{\S}}(\Omega^{1} (t_{\S}(M)))\rightarrow &t_{\S}F^{n-2}_{t_{\S}}(\Omega^{1} (t_{\S}(M)))\rightarrow \top t_{\S}F^{n-2}_{t_{\S}}(\Omega^{1} (t_{\S}(M))) \rightarrow 0.
\end{align*}
By \cite[Lemma 6.3]{HLM2}, we have $\ell\ell^{t_{\S}}(\Omega^1 (t_{\S}(M)))\leq \ell\ell^{t_{\S}}(\La)-1=n-1$.
It follows from Lemma \ref{lem-3.18} that $\ell\ell^{t_{\S}}(F^{n-1}_{t_{\S}}\Omega^{1}(t_{\S}(M)))=0$, that is,
$t_{\S}(F^{n-1}_{t_{\S}}\Omega^{1}(t_{\S}(M)))=0$.
Then by \cite[Proposition 5.3]{HLM2}, we have $\pd F^{n-1}_{t_{\S}}\Omega^{1}(t_{\S}(M)) \leq \alpha$.

We have the following
$$\Omega^{\alpha+2}(M)\cong \Omega^{\alpha+2}(t_{\S}(M)),$$
$$\Omega^{\alpha+2}(t_{\S}(M))=\Omega^{\alpha+1}(\Omega^{1} (t_{\S}(M)))\cong\Omega^{\alpha+1}( t_{\S}(\Omega^{1} (t_{\S}(M)))), $$
$$0 \rightarrow \Omega^{\alpha+1}(F_{t_{\S}}(\Omega^{1} (t_{\S}(M))))\rightarrow \Omega^{\alpha+1}(t_{\S}(\Omega^{1} (t_{\S}(M))))\oplus P_{1}\rightarrow
\Omega^{\alpha+1}(\top t_{\S}(\Omega^{1} (t_{\S}(M)))) \rightarrow 0,\ \text{(exact)}$$
$$\Omega^{\alpha+1}(F_{t_{\S}}(\Omega^{1} (t_{\S}(M))))\cong \Omega^{\alpha+1}( t_{\S}F_{t_{\S}}(\Omega^{1} (t_{\S}(M)))), $$
$$0 \rightarrow \Omega^{\alpha+1}(F^{2}_{t_{\S}}(\Omega^{1} (t_{\S}(M))))\rightarrow \Omega^{\alpha+1}(t_{\S}F_{t_{\S}}(\Omega^{1} (t_{\S}(M)))) \oplus P_{2}
\rightarrow \Omega^{\alpha+1}(\top t_{\S}F_{t_{\S}}(\Omega^{1} (t_{\S}(M)))) \rightarrow 0,\ \text{(exact)}$$
$$\cdots \cdots \cdots \cdots \notag$$
$$\Omega^{\alpha+1}(F^{n-2}_{t_{\S}}(\Omega^{1} (t_{\S}(M))))\cong\Omega^{\alpha+1}( t_{\S}F^{n-2}_{t_{\S}}(\Omega^{1} (t_{\S}(M)))), $$
$$\Omega^{\alpha+1}(t_{\S}F^{n-2}_{t_{\S}}(\Omega^{1} (t_{\S}(M))))\oplus P_{n-1}\cong \Omega^{\alpha+1}( \top t_{\S}F^{n-2}_{t_{\S}}(\Omega^{1} (t_{\S}(M)))),$$
where all $P_{i}$ are projective in $\mod \La$; we also have the following
$$\Omega^{-\alpha-2}(\Omega^{\alpha+2}(M))\cong
\Omega^{-\alpha-2}(\Omega^{\alpha+2}(t_{\S}(M)))=\Omega^{-\alpha-2}(\Omega^{\alpha+1}(\Omega^{1} (t_{\S}(M))))
\cong\Omega^{-\alpha-2}(\Omega^{\alpha+1}( t_{\S}(\Omega^{1} (t_{\S}(M))))),$$
$$0 \rightarrow \Omega^{-\alpha-2}(\Omega^{\alpha+1}(F_{t_{\S}}(\Omega^{1} (t_{\S}(M)))))\rightarrow
\Omega^{-\alpha-2}(\Omega^{\alpha+1}(t_{\S}(\Omega^{1} (t_{\S}(M)))))\oplus \Omega^{-\alpha-2}(P_{1})\oplus E_{1}$$
$$\rightarrow
\Omega^{-\alpha-2}(\Omega^{\alpha+1}(\top t_{\S}(\Omega^{1} (t_{\S}(M))))) \rightarrow 0,\ \text{(exact)}$$
$$\Omega^{-\alpha-2}(\Omega^{\alpha+1}(F_{t_{\S}}(\Omega^{1} (t_{\S}(M)))))\cong
\Omega^{-\alpha-2}(\Omega^{\alpha+1}( t_{\S}F_{t_{\S}}(\Omega^{1} (t_{\S}(M))))), $$
$$0 \rightarrow \Omega^{-\alpha-2}(\Omega^{\alpha+1}(F^{2}_{t_{\S}}(\Omega^{1} (t_{\S}(M)))))\rightarrow
\Omega^{-\alpha-2}(\Omega^{\alpha+1}(t_{\S}F_{t_{\S}}(\Omega^{1} (t_{\S}(M))))) \oplus  \Omega^{-\alpha-2}(P_{2})\oplus E_{2}
$$$$\rightarrow
\Omega^{-\alpha-2}(\Omega^{\alpha+1}(\top t_{\S}F_{t_{\S}}(\Omega^{1} (t_{\S}(M))))) \rightarrow 0,\ \text{(exact)}$$
$$\cdots \cdots \cdots \cdots \notag$$
$$\Omega^{-\alpha-2}(\Omega^{\alpha+1}(F^{n-2}_{t_{\S}}(\Omega^{1} (t_{\S}(M)))))
\cong\Omega^{-\alpha-2}(\Omega^{\alpha+1}( t_{\S}F^{n-2}_{t_{\S}}(\Omega^{1} (t_{\S}(M))))), $$
$$\Omega^{-\alpha-2}(\Omega^{\alpha+1}(t_{\S}F^{n-2}_{t_{\S}}(\Omega^{1} (t_{\S}(M)))))\oplus \Omega^{-\alpha-2}(P_{n-1})\cong
\Omega^{-\alpha-2}(\Omega^{\alpha+1}( \top t_{\S}F^{n-2}_{t_{\S}}(\Omega^{1} (t_{\S}(M))))),$$
where all $E_{i}$ are injective in $\mod \La$. So
\begin{align*}
&\Omega^{-\alpha-2}(\Omega^{\alpha+2}(M))\\
\cong&\Omega^{-\alpha-2}(\Omega^{\alpha+1}( t_{\S}\Omega^{1}(t_{\S}(M)))) \\
\in& \langle \Omega^{-\alpha-2}(\Omega^{\alpha+1}  ( F_{t_{\S}}\Omega^{1}(t_{\S}(M)))) \rangle_{1} \diamond
\langle \Omega^{-\alpha-2}(\Omega^{\alpha+1}  ( \top t_{\S}\Omega^{1}(t_{\S}(M)))) \rangle_{1}\\
\subseteq& \langle \Omega^{-\alpha-2}(\Omega^{\alpha+1}  ( F_{t_{\S}}\Omega^{1}(t_{\S}(M)))) \rangle_{1} \diamond
\langle \Omega^{-\alpha-2}(\Omega^{\alpha+1}  ( \Lambda/\rad\La)) \rangle_{1}\\
=& \langle \Omega^{-\alpha-2}(\Omega^{\alpha+1}  (t_{\S} F_{t_{\S}}\Omega^{1}(t_{\S}(M)))) \rangle_{1} \diamond
\langle \Omega^{-\alpha-2}(\Omega^{\alpha+1}  ( \Lambda/\rad\La)) \rangle_{1}\\
\subseteq& \langle \Omega^{-\alpha-2}(\Omega^{\alpha+1}(F^{2}_{t_{\S}}\Omega^{1}(t_{\S}(M))))\rangle_{1} \diamond
\langle \Omega^{-\alpha-2}(\Omega^{\alpha+1}  ( \Lambda/\rad\La)) \rangle_{1} \diamond  \langle
\Omega^{-\alpha-2}(\Omega^{\alpha+1} ( \Lambda/\rad\La)) \rangle_{1}\\
&\;\;\;\;\vdots\\
\subseteq& \langle \Omega^{-\alpha-2}(\Omega^{\alpha+1}(F^{n-2}_{t_{\S}}\Omega^{1}(t_{\S}(M))))\rangle_{1} \diamond
\underbrace{\langle \Omega^{-\alpha-2}(\Omega^{\alpha+1}  ( \Lambda/\rad\La)) \rangle_{1}\diamond \cdots \diamond
\langle \Omega^{-\alpha-2}(\Omega^{\alpha+1}  ( \Lambda/\rad\La)) \rangle_{1}}_{n-2}\\
=& \langle \Omega^{-\alpha-2}(\Omega^{\alpha+1}(F^{n-2}_{t_{\S}}\Omega^{1}(t_{\S}(M))))\rangle_{1} \diamond
\langle \Omega^{-\alpha-2}(\Omega^{\alpha+1}  ( \Lambda/\rad\La))  \rangle_{n-2}\\
\subseteq& \langle \Omega^{-\alpha-2}(\Omega^{\alpha+1}(t_{\S}F^{n-2}_{t_{\S}}\Omega^{1}(t_{\S}(M)))
 \oplus P_{n-1})\rangle_{1} \diamond  \langle \Omega^{-\alpha-2}(\Omega^{\alpha+1} (\Lambda/\rad\La)) \rangle_{n-2}\\
=& \langle  \Omega^{-\alpha-2}(\Omega^{\alpha+1}(\top
t_{\S}F^{n-2}_{t_{\S}}\Omega^{1}(t_{\S}(M))))\rangle_{1} \diamond
\langle \Omega^{-\alpha-2}(\Omega^{\alpha+1} (\Lambda/\rad\La)) \rangle_{n-2}\\
\subseteq& \langle \Omega^{-\alpha-2}(\Omega^{\alpha+1}  ( \Lambda/\rad\La))\rangle_{1} \diamond  \langle
\Omega^{-\alpha-2}(\Omega^{\alpha+1}  ( \Lambda/\rad\La))  \rangle_{n-2}\\
=& \langle \Omega^{-\alpha-2}(\Omega^{\alpha+1}  ( \Lambda/\rad\La))  \rangle_{n-1},
\end{align*}
and hence
\begin{align*}
M\in &\langle\oplus_{i=0}^{-\alpha-1} \Omega^{i}(\Lambda) \rangle_{\alpha+2}
\diamond\langle \Omega^{-\alpha-2}(\Omega^{\alpha+2}(M))\rangle_{1}\\
\subseteq&\langle\oplus_{i=0}^{-\alpha-1} \Omega^{i}(\Lambda) \rangle_{\alpha+2}
\diamond\langle \Omega^{-\alpha-2}(\Omega^{\alpha+1}  ( \Lambda/\rad\La))  \rangle_{n-1}\\
\subseteq&\langle(\oplus_{i=0}^{-\alpha-1} \Omega^{i}(\Lambda) )
\oplus\Omega^{-\alpha-2}(\Omega^{\alpha+1}  ( \Lambda/\rad\La))  \rangle_{\alpha+1+n}.\ \ \ \ \text{(by Corollary \ref{cor-2.3}(1))}
\end{align*}
It follows that
$$\mod \La=\langle(\oplus_{i=0}^{-\alpha-1} \Omega^{i}(\Lambda) )
\oplus\Omega^{-\alpha-2}(\Omega^{\alpha+1} (\Lambda/\rad\La))  \rangle_{\alpha+1+n}$$
and $\dim \La \leq \alpha+n$.
\end{proof}

As an application of Theorem \ref{thm-3.19}, we have the following

\begin{corollary}\label{cor-3.20}
\begin{enumerate}
\item[]
\item[(1)] {\rm (\cite[Example 1.6(ii)]{BA08S})} $\dim\mod\La\leq\LL(\La)-1$;
\item[(2)] {\rm (cf. Corollary \ref{cor-3.6} and \cite[4.5.1(3)]{Iya03R})} $\dim\mod\La\leq\gldim \La$.
\end{enumerate}
\end{corollary}

\begin{proof}
(1) Let $\S=\emptyset$. Then $\pd \mathcal{S}=-1$ and the torsion pair $(\T_{\S}, \mathfrak{F}(\S))=(\mod \La, 0)$.
By \cite[Propposition 5.9(a)]{HLM2}, we have $t_{\S}(\La)=\La$ and $\ell\ell^{t_{\S}}(\La)=\LL(\La)$.
It follows from Theorem \ref{thm-3.19} that $\dim\mod\La\leq\LL(\La)-1$.

(2) Let $\S=\S^{<\infty}=\{\text{all simple modules in} \mod\La\}$. Then $\pd \mathcal{S}=\gldim\La$
and the torsion pair $(\T_{\S}, \mathfrak{F}(\S))=(0,\mod \La)$.
By \cite[Propposition 5.3]{HLM2}, we have $t_{\S}(\La)=0$ and
$\ell\ell^{t_{\S}}(\La)=0$. It follows from Theorem \ref{thm-3.19} that $\dim\mod\La\leq\gldim\La$.
\end{proof}

By choosing some suitable $\mathcal{S}$ and applying Theorem \ref{thm-3.19},
we may obtain more precise upper bounds for $\dim \mod \La$
than that in Corollary \ref{cor-3.20}.

\begin{example}\label{ex-3.21}
{\rm
Consider the bound quiver algebra $\Lambda=kQ/I$, where $k$ is a field and $Q$ is given by
$$\xymatrix{
&2n+1\\
{2n}&1\ar[l]^{\alpha_{2n}}\ar[u]^{\alpha_{2n+1}}\ar[r]^{\alpha_{1}} \ar[d]^{\alpha_{n+1}}
&2\ar[r]^{\alpha_{2}} &{3}\ar[r]^{\alpha_{3}}&\cdots \ar[r]^{\alpha_{n-1}}&{n}\\
&n+1\ar[r]^{\alpha_{n+2}}
&n+2\ar[r]^{\alpha_{n+3}}&n+3\ar[r]^{\alpha_{n+4}}&\cdots\ar[r]^{\alpha_{2n-1}}&2n-1
}$$
and $I$ is generated by
$\{\alpha_{i}\alpha_{i+1}\;| \;n+1\leq i\leq 2n-2\}$ with $n\geq 5$.
Then the indecomposable projective $\La$-modules are
$$\xymatrix@-1.0pc@C=0.1pt
{& &1\edge[d]\edge[ld]\edge[rd]\edge[rrd]&&&&   & 2\edge[d]  &&&&  & &&&&&   &&& &\\
&n+1  &2\edge[d]  &2n&2n+1&&  &3\edge[d]  &&&& & 3\edge[d] &&&& &j\edge[d] &&&  & &\\
P(1)=& &3\edge[d]   &&&&P(2)=&4\edge[d]  &&&&P(3)=&4\edge[d] &&&&P(j)=&j+1,&&&&P(l)=l,&\\
& &\vdots\edge[d]&&&&   &\vdots\edge[d]&&&& &\vdots\edge[d]&&&&  & &&& &\\
&&n,&&&&  &n,   &&&& &n,  &&&& &  &&&  &\\
&&  &&&&  &   &&&& &  &&&&  &  &&&  &\\
}$$
where $n+1 \leq j\leq 2n-2$, $2n-1 \leq l\leq 2n+1$ and $P(i+1)=\rad P(i)$
for any $2 \leq i\leq n-1$.

We have
\begin{equation*}
\pd S(i)=
\begin{cases}
n-1, &\text{if}\;\;i=1;\\
1,&\text{if} \;\;2 \leq  i\leq n-1;\\
0,&\text{if} \;\; i=n, 2n, 2n+1;\\
2n-1-i,&\text{if}\;\; n+1 \leq  i\leq 2n-1.
\end{cases}
\end{equation*}
So $\S^{<\infty}=\{$all simple modules in $\mod \La\}$.
Let $\S:=\{S(i)\mid 2\leq i \leq n\}(\subseteq\mathcal{S}^{<\infty})$
and $\S'$ be all the others simple modules in $\mod \La$. Then
$\pd\S=1$ and $\S'=\{ S(i)\mid i=1\text{ or }n+1 \leq i \leq 2n+1 \}$.
Because $\La=\oplus_{i=1}^{2n+1}P(i)$, we have
$$\ell\ell^{t_{\S}}(\La)=\max\{\ell\ell^{t_{\S}}(P(i)) \mid 1 \leq i  \leq 2n+1\}$$
by \cite[Lemma 3.4(a)]{HLM2}.

In order to compute $\ell\ell^{t_{\S}}(P(1))$, we need to find the least non-negative integer $i$
such that $t_{\S}F_{t_{\S}}^{i}(P(1))=0$.
Since $\top P(1)=S(1)\in \add \S'$, we have $t_{\S}(P(1))=P(1)$ by \cite[Proposition 5.9(a)]{HLM2}.
Thus
\begin{align*}\xymatrix@-1.0pc@C=0.1pt{
&F_{t_{\S}}(P(1))=\rad t_{\S}(P(1))=\rad (P(1))=S(n+1)\oplus P(2)\oplus S(2n) \oplus S(2n+1).\\
}\end{align*}
Since $\top S(n+1)=S(n+1)\in \add \mathcal{S}'$, we have $t_{\S}(S(n+1))=S(n+1)$ by \cite[Proposition 5.9(a)]{HLM2}.
Similarly, $t_{\S}(S(2n))=S(2n)$ and $t_{\S}(S(2n+1))=S(2n+1)$. Since
$P(2)\in  \mathfrak{F}(\S)$, we have $t_{\S}(P(2))=0$ by \cite[Proposition 5.3]{HLM2}. So
$$t_{\S}F_{t_{\S}}(P(1))=t_{\S}(S(n+1)\oplus P(2)\oplus S(2n) \oplus S(2n+1))
=S(n+1)\oplus S(2n) \oplus S(2n+1).$$
It follows that $$F_{t_{\S}}^{2}(P(1))=\rad t_{\S}F_{t_{\S}}(P(1))
=\rad(S(n+1)\oplus S(2n) \oplus S(2n+1))=0$$
and $t_{\S}F_{t_{\S}}^{2}(P(1))=0$, which implies $\ell\ell^{t_{\S}}(P(1))=2$.
Similarly, we have
\begin{equation*}
\ell\ell^{t_{\S}}(P(i))=
\begin{cases}
0,&\text{if} \;\;2\leq  i\leq n;\\
2,&\text{if} \;\;n+1\leq  i\leq 2n-2;\\
1,&\text{if} \;\;2n-1\leq  i\leq 2n+1.
\end{cases}
\end{equation*}
Consequently, we conclude that $\ell\ell^{t_{\S}}(\La)=\max\{\ell\ell^{t_{\S}}(P(i))\mid 1 \leq i  \leq 2n+1\}=2$.

(1) Because $\LL(\La)=n$ and $\gldim \La=n-1$, we have
$$\dim \mod \La \leq \min\{ \gldim \La, \LL(\La)-1\}=n-1$$ by Corollary \ref{cor-3.20}.

(2) By Theorem \ref{thm-3.19}, we have
$$\dim \mod \La \leq \pd \S+\ell\ell^{t_{\S}}(\La)=1+2=3.$$
The upper bound here is better than that in (1) since $n\geq 5$.}
\end{example}

\section{Ring extensions}


Let $\Lambda$ be a subring of a ring $\Gamma$ such that $\Lambda$ and $\Gamma$ have the same identity. Then $A$ is
called a {\bf ring extension} of $\Lambda$, and denoted by $\Gamma \geq \Lambda$.

\begin{definition}\label{4.1}
{\rm A ring extension $\Gamma \geq \Lambda$ is called
\begin{itemize}
\item[(1)] {\rm (\cite{HZYSJX12I})}
a {\bf weak excellent extension} if
\begin{itemize}
\item[(1.1)] $\Gamma$ is {\bf $\Lambda$-projective} (\cite{P77}); that is, for a submodule $N_{\Gamma}$ of $M_{\Gamma}$, if $N_{\Lambda}$ is a direct
summand of $M_{\Lambda}$, denoted by $N_{\Lambda} \mid M_{\Lambda}$, then $N_{\Gamma} \mid M_{\Gamma}$;
\item[(1.2)] $\Gamma$ is a {\bf finite extension} of $\Lambda$; that is, there exists a finite set $\{\gamma_1,\cdots,\gamma_n\}$ in $\Gamma$ such that
$\Gamma=\sum_{i=1}^{n}\gamma_i\Lambda$;
\item[(1.3)] $\Gamma_{\Lambda}$ is flat and $_{\Lambda}\Gamma$ is projective;
\end{itemize}
\item[(2)] (\cite{P77,B84}) an {\bf excellent extension} if it is a weak excellent extension
and $\Gamma_{\Lambda}$ and $_{\Lambda}\Gamma$ are free with a common basis
$\{\gamma_1, \cdots, \gamma_n\}$, such that $\Lambda \gamma_{i}=\gamma_{i}\Lambda$ for any $1\leq i\leq n$.
\item[(3)] {\rm (\cite{XCC04O})}
a {\bf left idealized extension} if $\rad \Lambda$ is a left ideal of $\Gamma$.
\end{itemize}}
\end{definition}





We have the following

\begin{theorem}\label{thm-4.2}
Let $\Gamma\supseteq \Lambda$ be artin algebras. Then we have
\begin{itemize}
\item[(1)] $\dim \mod \Lambda\geq\dim \mod \Gamma$ if $\Gamma \geq \Lambda$ is a weak excellent extension,
and $\dim \mod \Lambda=\dim \mod \Gamma$ if $\Gamma \geq \Lambda$ is an excellent extension;
\item[(2)] $\dim \mod \Lambda\leq \dim \mod \Gamma+2$ if $\Gamma \geq \Lambda$ is a left idealized extension.
\end{itemize}
\end{theorem}

\begin{proof}
(1) Let $\Gamma \geq \Lambda$ be a weak excellent extension
and $\dim\mod \Lambda=n$ and $T\in\mod \Lambda$ such that $\mod \Lambda=\langle T \rangle_{n+1}$.
Let $X\in\mod \Gamma\subseteq\mod \Lambda$. Since ${_{\Lambda}\Gamma}$ is projective, $-\otimes_{\Lambda}\Gamma$ is exact. So we have
$X\otimes_{\Lambda}\Gamma\in\langle (T\otimes_{\Lambda}\Gamma)_\Gamma\rangle_{n+1}$ by Lemma \ref{lem-2.4}.
Since $X_{\Gamma}\mid(X\otimes_{\Lambda}\Gamma)_{\Gamma}$ by \cite[Lemma 1.1]{XWM96O},
we have $X_{\Gamma}\in \langle (T\otimes_{\Lambda}\Gamma)_{\Gamma}\rangle_{n+1}$.
Thus $\mod \Gamma= \langle (T\otimes_{\Lambda}\Gamma)_{\Gamma}\rangle_{n+1}$ and $\dim \mod \Gamma \leq n$.

Now let $\Gamma \geq \Lambda$ be an excellent extension and
$\dim\mod \Gamma=n$ and $S\in\mod \Gamma\subseteq\mod \Lambda$ such that $\mod \Gamma=\langle S \rangle_{n+1}$.
Let $X_{\Lambda}\in\mod \Lambda$. Then there exists an exact sequence
$$0 \longrightarrow X_{1}\longrightarrow X\otimes_{\Lambda}\Gamma  \longrightarrow X_{2}\longrightarrow 0$$
in $\mod \Gamma$ with $X_{1}\in\langle S_{\Gamma}\rangle_{1}$ and $X_{2}\in\langle S_{\Gamma}\rangle_{n}$.
Note that it is also an exact sequence in $\mod \Lambda$. So $(X\otimes_{\Lambda}\Gamma)_{\Lambda}\in \langle S_{\Lambda}\rangle_{n+1}$.
Since $X_{\Lambda}\mid(X\otimes_{\Lambda}\Gamma)_{\Lambda}$, we have $X_{\Lambda}\in \langle S_{\Lambda}\rangle_{n+1}$. Thus
$\mod \Lambda =\langle S_{\Lambda}\rangle_{n+1}$ and $\dim\mod \Lambda \leq n$.

(2) Let $\dim \mod \Gamma=n$. Then $\wresoldim \mod \Gamma=n$ by Theorem \ref{thm-3.5}. Let $X\in\mod \Lambda$.
Since $\Omega^{2}_{\Lambda}(X)$ can be viewed as an $\Gamma$-module by \cite[Lemma 0.2]{XCC05E}, there exists $V\in\mod \Gamma \subseteq\mod \Lambda$
such that there is an exact sequence
$$0 \longrightarrow V_{n} \longrightarrow V_{n-1} \longrightarrow\cdots
\longrightarrow V_{1 }
\longrightarrow V_{ 0} \longrightarrow \Omega^{2}_{\Lambda}(X)\longrightarrow 0 $$
in $\mod \Gamma$ with all $V_{i}$ in $\add V_\Gamma$. It is also an exact sequence in $\mod \Lambda$. So
$(V_{\Lambda}\oplus \Lambda)$-$\wresoldim\mod \Lambda\leq n+2$ and $\wresoldim \mod \Lambda \leq n+2$.
Thus $\dim \mod \Lambda \leq n+2$ by Theorem \ref{thm-3.5}.
\end{proof}

In the following, we list some examples of (weak) excellent extensions, in which Theorem \ref{thm-4.2}(1) may be applied.

\begin{example} {\rm (\cite{P77,B84,HZYSJX13E,X94})\label{ex-4.3}
\begin{enumerate}
\item[(1)] For a ring $\Lambda$, $M_n(\Lambda)$ (the matrix ring of $\Lambda$
of degree $n$) is an excellent extension of $\Lambda$.
\item[(2)] Let $\Lambda$ be a ring and $G$ a finite group. If $|G|^{-1}\in \Lambda$, then the skew group ring $\Lambda *G$
is an excellent extension of $\Lambda$.
\item[(3)] Let $\Lambda$ be a finite-dimensional algebra over a field $k$, and let $F$ be a finite separable field extension of $k$.
Then $\Lambda\otimes_{k}F$ is an excellent extension of $\Lambda$.
\item[(4)] Let $k$ be a field, and let $G$ be a group and $H$ a normal subgroup of $G$.
If $[G:H]$ is finite and is not zero in $k$, then $kG$ is an excellent extension of $kH$.
\item[(5)] Let $k$ be a field of charactertistic $p$, and let $G$ a finite group and $H$ a normal subgroup of $G$.
If $H$ contains a Sylow $p$-subgroup of $G$, then $kG$ is an excellent extension of $kH$.
\item[(6)] Let $k$ be a field and $G$ a finite group. If $G$ acts on $k$ (as field automorphisms) with kernel $H$.
Then the skew group ring $k*G$ is an excellent extension of the group ring $kH$, and the center
$Z(kH)$ of $kH$ is an excellent extension of the center $Z(k*G)$ of $k*G$.
\item[(7)] Let $H$ be a finite-dimensional semisimple Hopf algebra over a field $k$ and $\Lambda$ a twisted $H$-module algebra.
Then for any cocycle $\sigma\in \Hom_{k}(H \otimes H, \Lambda)$, the crossed product algebra $\Lambda\#_{\sigma}H$ is a weak excellent
extension of $\Lambda$, but not an excellent extension of $\Lambda$ in general.
\item[(8)] Recall from \cite{PR04} that a ring $\Lambda$ is called a {\bf right $S$-ring} if any flat module in $\mod \Lambda$ is projective.
The class of right $S$-rings includes semiperfect rings, commutative semilocal rings, subrings of right noetherian rings, subrings
of right $S$-rings, right Ore domains, right nonsingular ring of finite right Goldie dimension, endomorphism rings of right artinian
modules and rings with right Krull dimension (\cite{FHR03,PR04}). Let $\Gamma \geq \Lambda$ be an excellent extension with $\Lambda$ a right $S$-ring.
If $\Gamma$ has two ideals $I$ and $J$ such that $\Lambda\cap I=0$ and $\Gamma=I\oplus J$,
then the canonical embedding $\Lambda\hookrightarrow \Gamma/I$ is a weak excellent extension; and it is not an excellent extension
if $J_{\Lambda}$ is not free.
\end{enumerate}}
\end{example}

We recall from \cite{ML11F} the separable equivalence of artin algebras, which includes
the derived equivalence of self-injective algebras, Morita equivalence and stable equivalence
(of Morita type) (\cite{ML11F,PS17S}).

\begin{definition}{\rm (\cite{ML11F})}\label{def-4.4}
{\rm Two artin algebras $\Lambda$ and $\Gamma$ are called {\bf separably equivalent} if there exist $_{\Gamma}M_{\Lambda}$
and $_{\Lambda}N_{\Gamma}$ such that
\begin{itemize}
\item[(1)] $M$ and $N$ are both finitely generated projective as one sided modules;
\item[(2)] $M\otimes_{\Lambda}N\cong \Gamma\oplus X$ as a $(\Gamma,\Gamma)$-bimodule for some $_{\Gamma}X_{\Gamma}$;
\item[(3)] $N\otimes_{\Gamma}M\cong \Lambda\oplus Y$ as a $(\Lambda,\Lambda)$-bimodule for some $_{\Lambda}Y_{\Lambda}$.
\end{itemize}}
\end{definition}

We have the following

\begin{theorem}\label{thm-4.5}
Let $\Lambda$ and $\Gamma$ be artin algebras. If they are separably equivalent, then $\dim \mod \Lambda=\dim \mod \Gamma$.
\end{theorem}

\begin{proof}
Let $M$ and $N$ be as in Definition \ref{def-4.4}.
Let $\dim\mod \Gamma=n$. Then there exists $T_{\Gamma}\in \mod \Gamma$ such that $\mod \Gamma=\langle T_{\Gamma} \rangle_{n+1}$.
Let $L_{\Lambda}\in \mod \Lambda$. Then $L\otimes_{\Lambda}N_{\Gamma}\in \mod \Gamma=\langle T_{\Gamma} \rangle_{n+1}$.
Since $_{\Gamma}M$ is projective in $\Gamma$-$\mod$, we have that the
functor $-\otimes_{\Gamma}M: \mod \Gamma \longrightarrow \mod \Lambda$ is exact. By Lemma \ref{lem-2.4}, we have
$(L\otimes_{\Lambda}N)\otimes_{\Gamma}M\in \langle T\otimes_{\Gamma}M_{\Lambda} \rangle_{n+1}$.
By Definition \ref{def-4.4}(3), there exists a $(\Lambda,\Lambda)$-bimodule $Y$ such that
\begin{align*}
L\oplus (L\otimes_{\Lambda}Y)&\cong( L\otimes_{\Lambda}\Lambda)\oplus (L\otimes_{\Lambda}Y)\\
&\cong L\otimes_{\Lambda}(\Lambda\oplus Y)\\
&\cong L\otimes_{\Lambda}(N\otimes_{\Gamma}M)\\
&\cong (L\otimes_{\Lambda}N)\otimes_{\Gamma}M\\
&\in \langle T\otimes_{\Gamma}M_{\Lambda} \rangle_{n+1},
\end{align*}
and so $L_{\Lambda}\in  \langle T\otimes_{\Gamma}M_{\Lambda} \rangle_{n+1}$. It follows that $\mod \Lambda=\langle T\otimes_{\Gamma}M_{\Lambda} \rangle_{n+1}$
and $\dim \mod \Lambda\leq n=\dim \mod \Gamma$. Symmetrically, we have $\dim \mod \Gamma \leq \dim \mod \Lambda$.
\end{proof}

As a consequence of Theorem \ref{thm-4.5}, we have the following

\begin{corollary}\label{cor-4.6}
Let $\Lambda,\Gamma$ and $\Delta$ be finite dimensional algebras over a field $k$.
If $\Lambda$ is separably equivalent to $\Gamma$, then
$\dim \mod \Lambda\otimes_{k}\Delta=\dim \mod \Gamma\otimes_{k}\Delta$.
\end{corollary}

\begin{proof}
If $\Lambda$ is separably equivalent to $\Gamma$, then $\Lambda\otimes_{k}\Delta$ is separably equivalent to $\Gamma\otimes_{k}\Delta$
by {\rm \cite[p.227, Proposition]{PS17S}}. The assertion follows from Theorem \ref{thm-4.5}.
\end{proof}



\section{Recollements}

We recall the notion of recollements of abelian categories.

\begin{definition}{\rm (\cite{VFTP04C})\label{def-5.1}
A {\bf recollement}, denoted by ($\mathcal{A},\mathcal{B},\mathcal{C}$), of abelian categories is a diagram
$$\xymatrix{\mathcal{A}\ar[rr]!R|{i_{*}}&&\ar@<-2ex>[ll]!R|{i^{*}}\ar@<2ex>[ll]!R|{i^{!}}\mathcal{B}
\ar[rr]!L|{j^{*}}&&\ar@<-2ex>[ll]!L|{j_{!}}\ar@<2ex>[ll]!L|{j_{*}}\mathcal{C}}$$
of abelian categories and additive functors such that
\begin{enumerate}
\item[(1)] ($i^{*},i_{*}$), ($i_{*},i^{!}$), ($j_{!},j^{*}$) and ($j^{*},j_{*}$) are adjoint pairs;
\item[(2)] $i_{*}$, $j_{!}$ and $j_{*}$ are fully faithful;
\item[(3)] $\Ima i_{*}=\Ker j^{*}$.
\end{enumerate}}
\end{definition}

We list some properties of recollements of abelian categories (see \cite{VFTP04C,PC14H,PCVJ14R}), which will be useful later.

\begin{lemma}\label{lem-5.2}
Let ($\mathcal{A},\mathcal{B},\mathcal{C}$) be a recollement of abelian categories. Then we have
\begin{enumerate}
\item[(1)] $i^{*}j_{!}=0=i^{!}j_{*}$;
\item[(2)] the functors $i_{*}$ and $j^{*}$ are exact, $i^{!}$ and  $j_{*}$ are left exact, and $i^{*}$ and $j_{!}$ are right exact;
\item[(3)] the functors $i^{*}$, $i^{!}$ and $j^{*}$ are dense;
\item[(4)] all the natural transformations $\xymatrix@C=15pt{i^{*}i_{*}\ar[r]&1_{\mathcal{A}},}$
$\xymatrix@C=15pt{1_{\mathcal{A}}\ar[r]&i^{!}i_{*},}$
$\xymatrix@C=15pt{1_{\mathcal{C}}\ar[r]&j^{*}j_{!}}$
and $\xymatrix@C=15pt{j^{*}j_{*}\ar[r]&1_{\mathcal{C}}}$ are natural isomorphisms;
\item[(5)] for any object $B\in \mathcal{B}$,
\begin{itemize}
\item[(a)] if $i^{*}$ is exact, there is an exact sequence
$$\xymatrix@C=15pt{0\ar[r]&j_{!}j^{*}(B)\ar[r]^-{\epsilon_{B}}&
B\ar[r]&i_{*}i^{*}(B)\ar[r]&0}$$
\item[(b)] if $i^{!}$ is exact, there is an exact sequence
$$\xymatrix@C=15pt{0\ar[r]&i_{*}i^{!}(B)\ar[r]&B\ar[r]^-{\eta_{B}}&
j_{*}j^{*}(B)\ar[r]&0}.$$
\end{itemize}
\end{enumerate}
\end{lemma}

\begin{lemma}\label{lem-5.3}
Let $(\mathcal{A},\mathcal{B},\mathcal{C})$ be a recollement of abelian categories. Then we have
\begin{itemize}
\item[(1)] If $i^{*}$ is exact, then $j_{!}$ is exact;
\item[(2)] If $i^{!}$ ie exact, then $j_{*}$ is exact.
\end{itemize}
\end{lemma}

\begin{proof}
(1) Let
$$\xymatrix@C=15pt{0\ar[r]&X\ar[r]&Y\ar[r]&Z\ar[r]&0}$$ be an exact sequence in $\mathcal{C}$.
Since $j_{!}$ is right exact by Lemma \ref{lem-5.2}(2), we get an exact sequence
\begin{align}\label{2.1}
\xymatrix@C=15pt{0\ar[r]&C\ar[r]&j_{!}(X)\ar[r]&j_{!}(Y)\ar[r]&j_{!}(Z)\ar[r]&0}
\end{align}
in $\mathcal{B}$. Notice that $j^{*}$ is exact and $j^{*}j_{!}\cong 1_{\mathcal{C}}$ by Lemma \ref{lem-5.2}(2)(4),
so $j^{*}(C)=0$. Since $\Ima i_{*}=\Ker j^{*}$, there exists $C'\in\mathcal{A}$ such that $C\cong i_{*}(C')$.
Since $i^{*}$ is exact and $i^{*}j_{!}=0$ by Lemma \ref{lem-5.2}(2)(1), applying the functor $i^{*}$ to the
exact sequence (\ref{2.1}) yields $i^{*}(C)=0$. It follow that $C'\cong i^{*}i_{*}(C')\cong i^{*}(C)=0$ and $C=0$.
Thus $j_{!}$ is exact.

(2) It is dual to (1).
\end{proof}

Let $F:\mathcal{C}\rightarrow \mathcal{D}$ be a functor of additive categories. Recall from \cite{XCC06A}
that $F$ is called {\bf quasi-dense} if for any $D\in \mathcal{D}$, there exists $C\in \C$ such that $D$
is isomorphic to a direct summand of $F(C)$. Obviously, any dense functor is quasi-dense.

\begin{lemma}\label{lem-5.4}
Let $F:\mathcal{A}\rightarrow \mathcal{B}$ be an exact functor of abelian categories, and let
$\mathcal{A}_{1}$ and $\mathcal{B}_{1}$ be subcategories of $\mathcal{A}$ and $\mathcal{B}$ respectively.
If the restriction functor $F:\mathcal{A}_{1}\rightarrow \mathcal{B}_{1}$
is quasi-dense, then ${\bf size}_{\mathcal{A}}\mathcal{A}_{1} \geq {\bf size}_{\mathcal{B}}\mathcal{B}_{1}$;
in particular, $\dim\mathcal{A}\geq \dim \mathcal{B}$.
\end{lemma}

\begin{proof}
Suppose $\size_{\mathcal{A}}\mathcal{A}_{1}=n$, that is, $\mathcal{A}_{1}\subseteq\langle T\rangle_{n+1}$ for some $T\in \mathcal{A}$.
Let $X\in\mathcal{B}_{1}$. Since $F$ is quasi-dense, we have $X\oplus X_{1}\cong F(Y)$ for some $Y\in\mathcal{A}_{1}$ and $X_{1}\in\mathcal{B}_{1}$.
It follows from Lemma \ref{lem-2.4} that $X\oplus X_{1}\in F(\mathcal{A}_{1})\subseteq  F(\langle T\rangle_{n+1})\subseteq\langle F(T)\rangle_{n+1}$.
So $X\in\langle F(T)\rangle_{n+1}$ and $\mathcal{B}_{1}\subseteq\langle F(T)\rangle_{n+1}$, which implies
$\size_{\mathcal{B}}\mathcal{B}_{1}\leq n$.
\end{proof}

Let $\La$ be an artin algebra and $e$ an idempotent of $\La$. Then $(\mod \La/e\La e,\mod \La,\mod e\La e)$ is a recollement
by \cite[Example 2.7]{PC14H}. So $\dim \mod \La\geq \dim \mod e\La e$ by Lemma \ref{lem-5.4}.

\begin{theorem}\label{thm-5.5}
Let $(\mathcal{A},\mathcal{B},\mathcal{C})$ be a recollement of abelian categories. If either $i^{!}$ or $i^{*}$
is exact, then
$$\max\{\dim \mathcal{A},\dim \mathcal{C}\} \leq \dim \mathcal{B}\leq \dim \mathcal{A} +\dim \mathcal{C}+1.$$
\end{theorem}

\begin{proof}
Let $i^{!}$ be exact. Since $i^{!}$ and $j^{*}$ are exact and dense Lemma \ref{lem-5.2}(2)(3),
it follows from Lemma \ref{lem-5.4} that $\max\{\dim A,\dim C\} \leq \dim B$.

Let $\dim \mathcal{A}=n$ and $\dim \mathcal{C}=m$. Then there exist $X\in\mathcal{A}$ and $Y\in\mathcal{C}$
such that $\mathcal{A}=\langle X\rangle_{n+1}$ and $\mathcal{C}=\langle Y\rangle_{m+1}$. Let $M\in\mathcal{B}$.
Since $i^{!}$ is exact by assumption, we have an exact sequence
$$\xymatrix@C=15pt{0\ar[r]&i_{*}i^{!}(M)\ar[r]&M\ar[r]&
j_{*}j^{*}(M)\ar[r]&0}$$
in $\mathcal{B}$. Note that $i_{*}$ and $j_{*}$ are exact by Lemmas \ref{lem-5.2}(2) and \ref{lem-5.3}(2).
Since $i^{!}(M)\in \mathcal{A}=\langle X\rangle_{n+1}$ and $j^{*}(M)\in\mathcal{C}=\langle Y\rangle_{m+1}$,
we have $i_{*}i^{!}(M)\in\langle i_{*}(X)\rangle_{n+1}$ and $j_{*}j^{*}(M)\in\langle j_{*}(Y)\rangle_{m+1}$
by Lemma \ref{lem-2.4}. Thus $M\in \langle i_{*}X\rangle_{n+1}\diamond \langle j_{*}Y\rangle_{m+1} \subseteq
\langle i_{*}X\oplus j_{*}Y\rangle_{n+m+2}$ by Corollary \ref{cor-2.3}(1), and therefore $\dim \mathcal{B}\leq n+m+1$.

For the case that $i^{*}$ is exact, the argument is similar.
\end{proof}

Let $\La,\La',\La''$ be artin algebras and $(\mod \La',\mod \La,\mod \La'')$ be a recollement.
If $\dim\mod\La=0$, then  $\dim\mod\La'=0=\dim\mod\La''$; that is, $\Lambda$ is of finite representation type implies that
so are $\Lambda'$ and $\Lambda''$ (\cite{PC14H}). Conversely,
if $\dim\mod\La'=0=\dim\mod\La''$, then $\dim\mod\La=0$ does not hold true in general. For example,
let $\Lambda'$ be the finite dimensional algebra given by the quiver $\cdot$ (a unique vertex without arrows) and $\Lambda''$
the finite dimensional algebra given by the quiver
$$
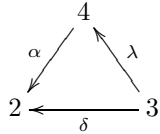
\xymatrix@C=15pt{&4\ar[ld]_{\alpha}\\
2&&\ar[ll]^{\delta}\ar[lu]_{\lambda}3}$$
with relation $\lambda\alpha=0$. Then both $\Lambda'$ and $\Lambda''$ are of finite representation type,
and so $\dim\mod\La'=0=\dim\mod\La''$ by \cite[Example 1.6(i)]{BA08S} (see Corollary \ref{cor-3.8}(1)).
Define the triangular matrix algebra $\Lambda:={\Lambda'\ M\choose 0\ \ \Lambda''}$,
where $M\cong \Lambda'\oplus \Lambda'$, the right $\Lambda''$-module structure on $M$ is induced by the unique algebra
surjective homomorphism $\xymatrix@C=15pt{\Lambda''\ar[r]^{\phi}&\Lambda'}$ satisfying $\phi(e_{2})=e_{1}$, $\phi(e_{3})=0$
and $\phi(e_{4})=0$. Then $\Lambda$ is the finite dimensional algebra given by the quiver
$$
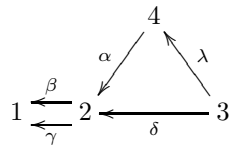
\xymatrix@C=15pt{&&4\ar[ld]_{\alpha}\\
1&\ar@<-1ex>[l]_{\beta}\ar@<1ex>[l]^{\gamma}2&&\ar[ll]^{\delta}
\ar[lu]_{\lambda}3}$$
with relations $\delta\gamma=\delta\beta=\lambda\alpha=\alpha\beta=\alpha\gamma=0$.
By {\rm \cite[Example 2.12]{PC14H}}, we have that
$$\xymatrix{\mod \Lambda'\ar[rr]!R|-{i_{*}}&&\ar@<-2ex>[ll]!R|-{i^{*}}
\ar@<2ex>[ll]!R|-{i^{!}}\mod \Lambda
\ar[rr]!L|-{j^{*}}&&\ar@<-2ex>[ll]!L|-{j_{!}}\ar@<2ex>[ll]!L|-{j_{*}}
\mod \Lambda''}$$
is a recollement, where
\begin{align*}
&i^{*}({X\choose Y}_{f})=\Coker f, & i_{*}(X)={X\choose 0},&&i^{!}({X\choose Y}_{f})=X,\\
&j_{!}(Y)={Y\choose Y}_{1}, & j^{*}({X\choose Y}_{f})=Y, &&j_{*}(Y)={0\choose Y}.
\end{align*}
Because $i^{!}$ is exact by \cite[Lemma 3.2(a)]{LM17G},
$\dim \mod\La\leq 1$ by Theorem \ref{thm-5.5}. Notice that $\La$ is of infinite representation type and
$\repdim \La=3$ by \cite[Example 5.9]{AICFU14O}, so $\dim \mod \La=1$ by Corollary \ref{cor-3.8}(2).

\vspace{0.6cm}

{\bf Acknowledgements.}
This work was partially supported by NSFC (No. 11571164), a Project Funded
by the Priority Academic Program Development of Jiangsu Higher Education Institutions, Postgraduate Research and
Practice Innovation Program of Jiangsu Province (Grant No. KYCX17\_0019).
The authors thank the referee for very useful and detailed suggestions.


%


\end{document}